\documentclass[english]{amsart}
\usepackage{dsfont}
\usepackage{url}
\usepackage[utf8]{inputenc}
\usepackage[T1]{fontenc}
\usepackage{lmodern}
\usepackage{babel}
\usepackage{mathtools}  
\usepackage{amssymb}
\usepackage{lipsum}
\usepackage{mathrsfs}
 \usepackage{mathabx}
\usepackage{color}

\newtheorem{theorem}{Theorem}[section]
\newtheorem{proposition}{Proposition}[section]
\newtheorem{lemma}{Lemma}[section]

\newtheorem{remark}{Remark}[section]

\numberwithin{equation}{section}

\title[inverse medium problem]{H\"older stability for an inverse medium problem\\ with  internal data}

\thanks{The authors are supported by the grant ANR-17-CE40-0029 of the French National Research Agency ANR (project MultiOnde).}
 
\author[Mourad Choulli]{Mourad Choulli}
\address{Mourad Choulli, IECL, UMR CNRS 7502, Universit\'e de Lorraine, Boulevard des Aiguillettes BP 70239 54506 Vandoeuvre Les Nancy cedex- Ile du Saulcy - 57 045 Metz Cedex 01 France}
\email{mourad.choulli@univ-lorraine.fr}

\author[Faouzi Triki]{Faouzi Triki}
\address{Faouzi Triki, Laboratoire Jean Kuntzmann,  UMR CNRS 5224, 
Universit\'e  Grenoble-Alpes, 700 Avenue Centrale,
38401 Saint-Martin-d'H\`eres, France}
\email{faouzi.triki@univ-grenoble-alpes.fr}

\begin{document}

\begin{abstract}

We are interested in an inverse medium problem with internal data. This problem is originated from multi-waves imaging.
 We aim in the present work to study the well-posedness of the inversion in terms of the boundary conditions. We precisely show that we have actually a stability estimate of H\"older type. For sake of simplicity, we limited our study to the class of Helmholtz equations $\Delta +V$ with bounded potential $V$.

\medskip
\noindent
{\bf Mathematics subject classification :} 35R30.

\smallskip
\noindent
{\bf Key words :} Helmholtz equation, inverse medium problem, internal data, H\"older stability, unique continuation.
\end{abstract}

\maketitle

\section{Introduction}

Let $\Omega$ be a $C^2$-smooth bounded domain of $\mathbb{R}^n$, $n=2,3$, with boundary $\Gamma$. 

Set
\[
\mathscr{D}=\{ V\in L^\infty (\Omega );\; \mbox{$0$ is not an eigenvalue of}\; A_V\},
\]
where $A_V:L^2(\Omega )\rightarrow L^2(\Omega )$ is the unbounded operator defined by 
\[
A_V=-\Delta -V\; \mbox{with}\; D(A_V)=H^2(\Omega )\cap H_0^1(\Omega).
\]
Note that
\[
\mathscr{D}\supset \{V;\; V(x)=\lambda ,\; x\in \Omega,\; \mbox{for some}\; \lambda \in \mathbb{R}\setminus \sigma (A_0)\}.
\]
Here $A_0$ is $A_V$ when $V=0$ and $\sigma (A_0)$ is the spectrum of $A_0$.

Let $0<k<1$ be given and let $\overline{V}\in \mathscr{D}$ so that $2v_0\le \overline{V}$, where $v_0>0$ is fixed. Consider then $\mathscr{D}_0(k,v_0,\overline{V})$, the subset of $\mathscr{D}$ of those functions $V\in L^\infty (\Omega )$ satisfying
\[
\|V-\overline{V}\|_{L^\infty (\Omega )}\le \min \left( k /\|A^{-1}_{\overline{V}}\|_{\mathscr{B}(L^2(\Omega))},v_0 \right).
\]

Pick $p>n$ and fix $h\in W^{2-1/p,p}(\Gamma )$ non identically equal to zero and denote by $u_V$, $V\in \mathscr{D}_0(k,v_0,\overline{V})$, the solution of the following BVP for the Helmholtz equation:
\[
\Delta u+Vu=0\; \mbox{in}\; \Omega \quad \mbox{and}\quad u=h\; \mbox{on}\; \Gamma .
\]
According to \cite[Theorem 3.1, page 1782]{CT}, $u_V\in W^{2,p}(\Omega )$ and the following estimate holds
\begin{equation}\label{1}
\|u_V\|_{W^{2,p}(\Omega )}\le M=M\left(\Omega ,p,v_0, k ,\overline{V}, h\right ),\quad V\in \mathscr{D}_0(k,v_0,\overline{V}).
\end{equation}

We are mainly interested in determining the absorption coefficient $V$ from the internal data
\[
I_V=Vu_V^2.
\]

This inverse problem is originated from multi-waves imaging. The term multi-waves refers to the fact that two  types of physical
waves are used to probe the medium under study. Usually, the first wave is sensitive to the contrast
of the desired parameter,  the other types can carry the information revealed by the first type of waves to the boundary of the medium  where measurements can be taken. In the present work, we assume that the first inversion has been performed, that is
the internal data  $I_V$  is retrieved, and we focus on the second step. We refer to \cite{ABCTF, ACGRT, AGNS, BK, BS, BU, CT} and reference therein for further details.

Choose $K>0$  sufficiently large in such a way that
\[
\mathscr{D}_1(k,v_0,\overline{V},K)=\mathscr{D}_0(k,v_0,\overline{V})\cap \{V\in C^{0,1}(\overline{\Omega});\; \|V\|_{C^{0,1}(\overline{\Omega})}\le K \}\ne\o.
\]
Note that, according to Rademacher's theorem, $C^{0,1}(\overline{\Omega})$ is continuously embedded in $W^{1,\infty}(\Omega )$.

For the inverse problem under consideration we are going to prove various H\"older stability estimates.

\begin{theorem}\label{theorem1.1} $(\textrm{interior stability})$
Let $\omega \Subset \Omega$. Then, there exist two constants $C=C\left(\Omega ,p,v_0, k ,\overline{V}, h ,K ,\omega \right )>0$ and $\mu =\mu \left(\Omega ,p,v_0, k ,\overline{V}, h ,K ,\omega \right )$ so that, for any $V,\tilde{V} \in \mathscr{D}_1(k,v_0,\overline{V},K)$
satisfying $V= \tilde V$ on $\Gamma$, we have
\[
\| V-\tilde{V}\|_{L^\infty (\omega )} \le C\left\| I_V^{1/2}-I_{\tilde{V}}^{1/2}\right\|_{H^1(\Omega )}^\mu.
\]
\end{theorem}

A similar result was already proved by G. Alessandrini in \cite{Al} under different assumptions. 

\smallskip
When $h$ does not vanish on  a part of 
$\Gamma$, we  obtain  improved results which we state in the following theorems.

\begin{theorem}\label{theorem1.2} 
Under the assumption  $|h|>\kappa >0$ on $\Gamma$, there exist two constants $C=C\left(\Omega ,p,v_0, k ,\overline{V}, h ,K\right )>0$ and $\mu =\mu \left(\Omega ,p,v_0, k ,\overline{V}, h ,K
 \right )$ so that, for any $V,\tilde{V} \in \mathscr{D}_1(k,v_0,\overline{V},K)$ satisfying
 $V= \tilde V$ on $\Gamma$, we have
\[
\| V-\tilde{V}\|_{L^\infty (\Omega )} \le C\left\| I_V^{1/2}-I_{\tilde{V}}^{1/2}\right\|_{H^1(\Omega )}^\mu.
\]
\end{theorem}

\begin{theorem}\label{theorem1.2b} Let $\tilde{\Omega}\subset  \Omega$ be a  $C^2$-smooth domain with boundary 
$\tilde{\Gamma}$ such that $\tilde{\gamma}:=\tilde{\Gamma} \cap \Gamma$ satisfies $\ring{\tilde{\gamma}} \ne\o$, and
let $\omega \Subset \tilde{\Omega}$. 
Under the assumption  $|h|>\kappa >0$ on $\Gamma$, there exist two constants $C=C\left(\Omega ,p,v_0, k ,\overline{V}, h ,K, \omega, \tilde{\Omega} \right )>0$ and $\mu =\mu 
\left(\Omega ,p,v_0, k ,\overline{V}, h ,K, \omega,  \tilde{\Omega} \right )$ so that, for any $V,\tilde{V} \in \mathscr{D}_1(k,v_0,\overline{V},K)$ 
satisfying $V=\tilde{V}$ on $\Gamma$ and  $\nabla V=  \nabla \tilde{V}$ on $\tilde{\gamma}$, we have
\[
\| V-\tilde{V}\|_{L^\infty (\omega )} \le C\left\| I_V^{1/2}-I_{\tilde{V}}^{1/2}\right\|_{H^1(\tilde{\Omega} )}^\mu.
\]
\end{theorem}

In the following result we allow the Dirichlet boundary condition  $h$  to vanish on $\Gamma$. Denote
\[
\Gamma_+=\{x\in \Gamma ;\; |h(x)|>0\}\quad \mbox{and}\quad \Gamma_0=\{ x\in \Gamma ;\; h(x)=0\}.
\]

\begin{theorem}\label{theorem1.3}
Let $\gamma$ be a compact subset of $\Gamma_+\cup \ring{\Gamma}_0$ and $\omega \Subset \Omega\cup \gamma$. Then, there exist two constants $C=C\left(\Omega ,p,v_0, k ,\overline{V}, h ,K ,\omega \right )>0$ and $\mu =\mu \left(\Omega ,p,v_0, k ,\overline{V}, h ,K ,\omega \right )$ so that, for any $V,\tilde{V} \in \mathscr{D}_1(k,v_0,\overline{V},K)$ satisfying $V= \tilde V$ on $\Gamma$, we have
\[
\| V-\tilde{V}\|_{L^\infty (\omega )} \le C\left\| I_V^{1/2}-I_{\tilde{V}}^{1/2}\right\|_{H^1(\Omega )}^\mu.
\]
\end{theorem}

We deduce from Theorem \ref{theorem1.2b}   that it is possible to recover the potential $V$ on a small  set $\omega$
of $\Omega$ if the medium is probed starting  from  $\tilde{\gamma}$,  the part of the boundary   where $h$ is intense, and
by covering a neighboring  region $\tilde{\Omega}$ that contains $\omega$. Unfortunately   $h$ may  in general settings
 be zero on some  parts of the boundary $\Gamma$, and Theorem \ref{theorem1.3}  is an attempt to 
 improve the results of Theorems \ref{theorem1.2} and   \ref{theorem1.2b} in this direction. 

In \cite{ADFV}, the authors were able to prove a lower bound for the gradient of solutions near the boundary when the boundary data is ``qualitatively unimodal'' (see \cite{ADFV} for the definition). Roughly speaking, the key in their proof is that even if the tangential gradient of solutions vanishes, there is, according to Hopf's maximum principle, a non zero contribution of the derivative of solutions in the normal direction. Unfortunately, there is no similar arguments that can be used in order to get lower bound for solutions on the boundary. 

The rest of this text consists in two sections and an appendix. We establish in Section 2 weighted interpolation inequalities that are our main ingredient for proving stability estimates. That is what we do in Section 3. We end by an appendix that is devoted to prove a technical result we used in Section 2. This result which is essential in our analysis gives a lower bound of the $L^2$-norm of solutions on a small ball, away from the boundary, in term of the radius of the ball.

\section{Weighted interpolation inequalities}

We start with some preliminaries involving the so-called frequency function. If $\mathbf{d}$ is the diameter of $\Omega$ with respect to the euclidean metric and $0<\delta <\mathbf{d}$, let
\[
\Omega ^\delta =\{x\in \Omega ;\; \mbox{dist}(x,\Gamma )\ge \delta\}
\quad \mbox{and}\quad
\Omega _\delta =\{x\in \Omega ;\; \mbox{dist}(x,\Gamma )\le \delta\}.
\]

For $0<v_0\le V_0$ and $M>0$, set  $\mathscr{V}(v_0,V_0)=\{V\in L^\infty (\Omega );\; v_0\le V\le V_0\}$. Define then
\[
\mathscr{S}(v_0,V_0)=\{ u\in H^2(\Omega );\; \; \Delta u+Vu=0,\; \mbox{for some}\; V\in \mathscr{V}(v_0,V_0)\}.
\]
Let  $u\in \mathscr{S}(v_0,V_0)$ and $x_0\in \Omega ^\delta$. We define, for $0<r<\delta$,
\begin{align*}
&H_u(x_0,r)=\int_{S(x_0,r)} u^2(x)dS(x),
\\
&D_u(x_0,r)=\int_{B(x_0,r)}\left\{|\nabla u(x)|^2+V(x)u^2(x)\right\}dx,
\\
&K_u(x_0,r)=\int_{B(x_0,r)} u^2(x)dS(x).
\end{align*}
Here $S(x_0,r)$ is the sphere of centrer $x_0$ and radius $r$.

Henceforth, the first Dirichlet eingenvalue of the Laplace operator in the domain $D$ is denoted by $\lambda_1(D)$.

Prior to define the frequency function, we need to prove the following lemma, where 
\begin{equation}\label{d1}
\rho_0=\sqrt{\lambda_1(B(0,1))/V_0}.
\end{equation}

\begin{lemma}\label{lemma+1}
Let $\rho_0$ be as in \eqref{d1}. Then, for any $u\in \mathscr{S}(v_0,V_0)$, $u\ne 0$, $x_0\in \Omega ^\delta$ and $0<r<\min(\delta ,\rho_0)$, we have $H_u(x_0,r)>0$.
\end{lemma}

\begin{proof}
We proceed by contradiction. Pick then $u\in \mathscr{S}(v_0,V_0)$, $u\ne 0$, $x_0\in \Omega ^\delta$ and assume then that $H_u(x_0,r)=0$ for some $0<r<\min(\delta ,\rho_0)$. That is, $u=0$ on $S(x_0,r)$. Therefore Green's formula gives
\begin{equation}\label{+1}
\int_{B(x_0,r)}|\nabla u|^2dx=\int_{B(x_0,r)}Vu^2dx\le V_0\int_{B(x_0,r)}u^2.
\end{equation}
On the other hand, bearing in mind that $\lambda_1(B(x_0,r))=\lambda_1(B(0,1))/r^2$, Poincar\'e's inequality yields
\begin{align}
\int_{B(x_0,r)}u^2dx&\le \frac{1}{\lambda_1(B(x_0,r))}\int_{B(x_0,r)}|\nabla u|^2dx\label{+2}
\\
&\le \frac{r^2}{\lambda_1(B(0,1))}\int_{B(x_0,r)}|\nabla u|^2dx.\nonumber
\end{align}
Note that we used here that $u\in H_0^1(B(x_0,r))$.

Inequality \eqref{+2} in \eqref{+1} produce
\[
\left( 1-\frac{r^2V_0}{\lambda_1(B(0,1))}\right)\int_{B(x_0,r)}|\nabla u|^2dx\le 0.
\]
As $1-r^2V_0/\lambda_1(B(0,1))>0$, we conclude that $u=0$ in $B(x_0,r)$ and hence $u$ is identically equal to zero by the unique continuation property. This leads to the expected contradiction.
\end{proof}

According to Lemma \ref{lemma+1}, if $u\in \mathscr{S}(v_0,V_0)$, $u\ne 0$, we can define the frequency function $N_u$, corresponding to $u$, by
\[
N_u(x_0,r)=\frac{rD_u(x_0,r)}{H_u(x_0,r)}.
\]

The following two lemmas can be deduced from the calculations developed in \cite{GL1, GL2} (see also \cite{Ch1}).
\begin{lemma}\label{lemma-a3}
For $u\in \mathscr{S}(v_0,V_0)$, $u\ne 0$ and $x_0\in \Omega ^\delta$, we have
\[
K_u(x_0,r)\le rH_u(x_0,r),\quad  0<r<\delta_0=\min (\rho_0,\rho_1,\delta ) ,
\]
where $\rho_0$ is as in \eqref{d1} and $\rho_1=\sqrt{(n-1)/V_0}$.
\end{lemma}

\begin{lemma}\label{lemma-a2}
Let $u\in \mathscr{S}(v_0,V_0)$, $u\ne 0$ and $x_0\in \Omega ^\delta$. Then
\[
N_u(x_0,r)\le C\max (N_u(x_0,\delta _0),1),\quad 0<r<\delta_0,
\]
with a constant $C=C(\Omega ,V_0)>0$ and $\delta_0$ is as in Lemma \ref{lemma-a3}.
\end{lemma}

Fix $0<\alpha \le 1$. We say that $W\subset L^1_+(\Omega )=\{ w\in L^1(\Omega );\; w\ge 0\}$ is a uniform set of weights for the weighted interpolation inequality 
\begin{equation}\label{w1}
\| f\|_{L^\infty (\Omega )}\le C\|f\|_{C^{0,\alpha}(\overline{\Omega})}^{1-\mu}\|fw\|_{L^1(\Omega )}^\mu ,
\end{equation}
if the constants $C>0$ and $0< \mu <1$ in \eqref{w1} can be chosen independently in $w\in W$ and $f\in C^{0,\alpha}(\overline{\Omega})$.

Similarly, we will say that $W\subset L^1_+(\Omega )$ is a uniform set of interior weights for the interior weighted interpolation inequality, where $\omega \Subset \Omega$ is arbitrary,
\begin{equation}\label{w1i}
\| f\|_{L^\infty (\omega )}\le C\|f\|_{C^{0,\alpha}(\overline{\Omega})}^{1-\mu}\|fw\|_{L^1(\Omega )}^\mu
\end{equation}
 if the constants $C>0$ and $0< \mu <1$ in \eqref{w1i}, depending on $\omega$, can be chosen independently in $w\in W$ and $f\in C^{0,\alpha}(\overline{\Omega})$.

\begin{remark}\label{remark1}
{\rm
Let $W\subset L^1_+(\Omega )$ be a uniform set of weights for the weighted interpolation inequality \eqref{w1}. Pick $w\in W$ and consider $Z=\{x\in \overline{\Omega};\; w(x)=0\}$. We claim that $Z$ has empty interior. Otherwise, if $\ring{Z}\ne\o$ then we would find $f_0\in C_0^\infty (\ring{Z})$ non identically equal to zero. But \eqref{w1} with $f=f_0$ would entail that $f_0=0$. That is we have a contradiction and our claim is proved. 
}
\end{remark}

Let $\varkappa$ be the norm of the imbedding $H^2(\Omega )\hookrightarrow C(\overline{\Omega})$. We introduce two sets, where $0<\kappa \le \varkappa M$ are given constants,
\begin{align*}
\mathscr{S}_w(v_0,V_0,\kappa ,M)=\{u\in H^2(\Omega );\; \Delta u+Vu&=0,\; \mbox{for some}\; V\in \mathscr{V}(v_0,V_0)\
\\ 
&\mbox{and}\; \|u\|_{H^2(\Omega)}\le M,\; \|u\|_{L^\infty (\Gamma )}\ge \kappa  \},
\end{align*}
and
\begin{align*}
\mathscr{S}_s(v_0,V_0,\kappa ,M)=\{u\in H^2(\Omega );\; \Delta u+Vu&=0,\; \mbox{for some}\; V\in \mathscr{V}(v_0,V_0)\
\\ 
&\mbox{and}\; \|u\|_{H^2(\Omega)}\le M,\; |u|\ge \kappa \; \mbox{on}\; \Gamma \}.
\end{align*}

The main ingredient in establishing weighted interpolation inequalities is the following theorem. Its proof, which is quite technical, is given in Appendix A.
\begin{theorem}\label{theorem2.1}
Let $0<\delta <\mathbf{d}$. Then there exists a constants $c=c(\Omega ,v_0,V_0,\kappa ,M,\delta )>0$ so that, for any  $x_0\in \Omega ^\delta$ and $u\in \mathscr{S}_w(v_0,V_0,\kappa ,M)$, we have
\begin{equation}\label{2.1.1}
e^{-e^{c/\delta}}\le \|u\|_{L^2(B(x_0,\delta ))}.
\end{equation}
\end{theorem}

\begin{theorem}\label{theorem-wii}
(1) The set $\mathscr{W}_w(v_0,V_0,\kappa ,M)=\{w=u^2;\; u\in \mathscr{S}_w(v_0,V_0,\kappa ,M)\}$ is a uniform set of interior weights for the interior weighted interpolation inequality \eqref{w1i}.

\noindent
(2) The set $\mathscr{W}_s(v_0,V_0,\kappa ,M)=\{w=u^2;\; u\in \mathscr{S}_s(v_0,V_0,\kappa ,M)\}$ is a uniform set of weights for the weighted interpolation inequality \eqref{w1}.
\end{theorem}

Before proving this theorem, we establish some preliminaries.

\begin{lemma}\label{lemma-a4}
Let $0<\delta <\mathbf{d}$. There exists a constant $C=C(\Omega,v_0,V_0,\kappa ,M,\delta )>0$ so that, for any $u \in \mathscr{S}_w(v_0,V_0,\kappa , M)$, we have
\[
\|N_u\|_{L^\infty (\Omega ^\delta \times (0,\delta_0))}\le C.
\]
Here $\delta _0$ is as in Lemma \ref{lemma-a3}.
\end{lemma}

\begin{proof}
Let $x_0\in \Omega ^\delta$. From Theorem \ref{theorem2.1}
\begin{equation}\label{a12}
K_u(x_0,\delta_0 )=\| u\|_{L^2(B(x_0,\delta_0 ) )}^2 \ge C.
\end{equation}
 Combined with Lemma \ref{lemma-a3}, this estimate yields
\begin{equation}\label{a13}
H_u(x_0,\delta_0 )\ge C.
\end{equation}
In light of Lemma \ref{lemma-a2}, we end up getting
\[
N_u(x_0,r)\le C,\quad 0<r<\delta_0,
\]
which leads immediately to the expected inequality.
\end{proof}

\begin{proposition}\label{proposition-a2}
Let $0<\delta <\mathbf{d}$. There exist two constants $C=C(\Omega,v_0,V_0,\kappa ,M,\delta )>0$ and $c=c(\Omega,v_0,V_0,\kappa ,M,\delta )>0$ so that, for any $u\in \mathscr{S}_w(v_0,V_0,\kappa ,M)$, we have
\[
Cr^c\le \|u\|_{L^2(B(x_0,r ))},\quad x_0\in \Omega ^\delta ,\;0<r<\delta_0,
\]
where $\delta _0$ is as in Lemma \ref{lemma-a3}.
\end{proposition}

\begin{proof}
Pick $u\in \mathscr{S}_w(v_0,V_0,\kappa ,M)$ and  $x_0\in \Omega ^\delta$. For simplicity's sake, set $H=H_u$ and $N=N_u$. From the calculations carried out in \cite{GL1,GL2} (see also \cite{Ch1}), we have
\[
\partial _rH(x_0,r)=\frac{n-1}{r}H(x_0,r)+2D(x_0,r).
\]
Whence
\[
\partial _r\left(\ln \frac{H(x_0,r)}{r^{n-1}}\right)=\frac{\partial _rH(x_0,r)}{H(x_0,r)}-\frac{n-1}{r}=\frac{2N(x_0,r)}{r}.
\]
This and Lemma \ref{lemma-a4} entail
\[
\partial _r\left(\ln \frac{H(x_0,r)}{r^{n-1}}\right)\le \frac{C}{r},\quad 0<r<\delta_0 .
\]
Thus
\[
\int_{sr}^{s\delta_0} \partial _t\left(\ln \frac{H(x_0,t)}{t^{n-1}}\right)dt=\ln \frac{H(x_0,s\delta_0)r^{n-1}}{H(x_0,sr)\delta_0 ^{n-1}} \le \ln \frac{\delta_0 ^C}{r^C},\quad 0<s<1,\; 0<r<\delta_0.
\]
Hence
\[
H(x_0,s\delta_0)\le \frac{C}{r^c}H(x_0, sr), \quad 0<r<\delta_0,
\]
and then
\begin{align*}
\|u\|_{L^2(B(x_0,\delta_0))}^2&= \delta ^{n-1}_0\int_0^1H(x_0,s\delta_0 )s^{n-1}ds
\\
&\le \frac{C}{r^c}r ^{n-1}\int_0^1H(x_0, sr )s^{n-1}ds= \frac{C}{r^c}\|u\|_{L^2(B(x_0,r ))}^2, \quad 0<r<\delta_0.
\end{align*}
Combined with \eqref{a12}, this estimate yields
\[
Cr^c\le \|u\|_{L^2(B(x_0,r ))}, \quad 0<r<\delta_0,
\]
as expected.
\end{proof}

\begin{proof}[Proof of Theorem \ref{theorem-wii}]
(1) Let $w\in \mathscr{W}_w(v_0,V_0,\kappa ,M)$ and $u\in \mathscr{S}_w(v_0,V_0,\kappa ,M)$ so that $w=u^2$. Fix $\omega \Subset \Omega$. We need to prove that \eqref{w1i} holds with constants $C$ and $\mu$ that are  independent of $w$ and $f\in C^{0,\alpha}(\overline{\Omega})$. By homogeneity it is enough to establish \eqref{w1i} when $\|f\|_{C^{0,\alpha}(\overline{\Omega})}=1$. To this end, take $f\in C^{0,\alpha}(\overline{\Omega})$ satisfying $\|f\|_{C^{0,\alpha}(\overline{\Omega})}=1$.

Fix $0<\delta <\mathbf{d}$ so that $\overline{\omega}\subset \Omega ^\delta$ and pick $x_0\in \overline{\omega}$ so that $|f(x_0)|=\|f\|_{L^\infty (\omega )}$. According to Proposition \ref{proposition-a2}, there exist   two constants $C=C(\Omega,v_0,V_0,\kappa ,M,\delta )>0$ and $c=c(\Omega,v_0,V_0,\kappa ,M,\delta )>0$ so that, for any $u\in \mathscr{S}_w(v_0,V_0,\kappa ,M)$, we have
\begin{equation}\label{wii1}
Cr^c\le \|u\|_{L^2(B(x_0,r ))},\quad 0<r<\delta_0 ,
\end{equation}
where $\delta_0$ is as in Lemma \ref{lemma-a3}.

But
\[
|f(x_0)|=\|f\|_{L^\infty (\omega )}\le |f(x)|+r^\alpha ,\;\; \mbox{for any}\; x\in B(x_0,r).
\]
Therefore
\begin{align*}
\|f\|_{L^\infty (\omega )}\int_{B(x_0,r)} u(x)^2dx &\le  2\int_{B(x_0,r)}|f(x)| u(x)^2dx +2r^\alpha \int_{B(x_0,r)} u(x)^2dx
\\
&\le 2\int_\Omega |f(x)| u(x)^2dx+2r^\alpha \int_{B(x_0,r)} u(x)^2dx.
\end{align*}
Note that, according to the unique continuation property, 
\[\int_{B(x_0,r)} u(x)^2dx\ne 0,\quad 0<r<\delta_0.\]
 Hence
\[
\|f\|_{L^\infty (\Omega )}\le2 \frac{\|fu^2\|_{L^1(\Omega )}}{\|u\|_{L^2(B(x_0,r))}^2}+2r^\alpha .
\]
Combined with \eqref{wii1}, this estimate yields
\begin{equation}\label{wii2}
\|f\|_{L^\infty (\omega )}\le C(\|fu^2\|_{L^1(\Omega )} r^{-c}+r^\alpha ),\;\; 0<r<\delta_0 .
\end{equation}
When $\|fu^2\|_{L^1(\Omega )}< \delta _0^{c+\alpha}$, we can take 
$r=\|fu^2\|_{L^1(\Omega )}^{1/(c+\alpha)}$ in \eqref{wii2} in order to get
\begin{equation}\label{wii3}
\|f\|_{L^\infty (\omega )}\le C\|fu^2\|_{L^1(\Omega )}^\mu,
\end{equation}
with $\mu =\frac{\alpha}{c+\alpha }$.

If $\|f u^2\|_{L^1(\Omega )}\ge \delta _0^{c+\alpha}$, we have
\begin{equation}\label{wii4}
\|f\|_{L^\infty (\omega )}\le \delta_0^{-(c+\alpha)}\|fu^2\|_{L^1(\Omega )}\le \delta_0^{-(c+\alpha)}M^{2-2\mu}\|fu^2\|_{L^1(\Omega )}^\mu .
\end{equation}
The expected inequality follows then from \eqref{wii3} and \eqref{wii4}.

(2) Let $w\in \mathscr{W}_s(v_0,V_0,\kappa ,M)$ and $u\in \mathscr{S}_s(v_0,V_0,\kappa ,M)$ so that $w=u^2$. As in (1), we have to prove that \eqref{w1} holds with constants $C$ and  $\mu$ independent on $w$ and $f\in C^{0,\alpha}(\overline{\Omega})$. As we have seen before, by homogeneity, it is enough to establish \eqref{w1} when $\|f\|_{C^{0,\alpha}(\overline{\Omega})}=1$. Let then $f\in C^{0,\alpha}(\overline{\Omega})$ satisfying $\|f\|_{C^{0,\alpha}(\overline{\Omega})}=1$.

Since $H^2(\Omega )$ is continuously embedded in $C^{0,1/2}(\overline{\Omega})$, there exists a constant $a=a(\Omega )>0$ so that
\[
[u]_{1/2}=\sup\left\{\frac{|u(x)-u(y)|}{|x-y|^{1/2}};\; x,y\in \overline{\Omega}\; x\neq y\right\}\le a\|u\|_{H^2(\Omega)}\le aM.
\]
Fix $\delta _1\le (\kappa/(2aM))^2$. Then a straightforward computation gives
\[
|u|\ge \kappa /2 \quad \mbox{in}\; \Omega_{\delta_1}.
\]
 
From the H\"older's continuity of $f$, we get, where $\eta=\delta _1/4$,
\[
\|f\|_{L^\infty (\Omega ^\eta )}=|f(x_0)|\le |f(x)|+r^\alpha ,\quad x\in B(x_0,r),\; 0<r<\delta_0 =\delta_0(\eta) .
\]
Whence, proceeding as in (1), we get
\[
\|f\|_{L^\infty (\Omega ^\eta )}\le \frac{\|fu^2\|_{L^1(\Omega )}}{\|u\|_{L^2(B(x_0,r))}^2}+r^\alpha .
\]
This and Proposition \ref{proposition-a2} yield
\begin{equation}\label{a15}
\|f\|_{L^\infty (\Omega ^\eta )}\le C(\|fu^2\|_{L^1(\Omega )} r^{-c}+r^\alpha ),\quad 0<r<\delta_0 .
\end{equation}

On the other hand, noting that $B(y,r)\cap \Omega\subset \Omega_{\delta_1}$ when $y\in \Omega_{2\eta}$ and $0<r<\eta$, we find
\begin{align*}
\|f\|_{L^\infty (\Omega _{2\eta} )}\int_{B(y_0,r)\cap \Omega} u(x)^2dx&=|f(y_0)|\int_{B(y_0,r)\cap \Omega} u(x)^2dx
\\
&\le \int_{B(y_0,r)\cap \Omega}|f(x)| u(x)^2dx +r^\alpha \int_{B(y_0,r)\cap \Omega} u(x)^2dx.
\end{align*}
We have
\[
\int_{B(y_0,r)\cap \Omega}u(x)^2dx\ge \kappa ^2|B(y_0,r)\cap \Omega | %(\kappa ^2r^n|B(0,1)|)/2,
\] 
Proceeding similarly to the proof of \cite[Appendix A]{CKa}, we find $\aleph =\aleph (\Omega )>0$ and $0<r_0=r_0(\Omega ,\kappa ,M)<\eta$ so that
\[
|B(y_0,r)\cap \Omega |\ge \aleph r^n,\quad 0<r<r_0.
\]
Whence
\[
\|f\|_{L^\infty (\Omega _{2\eta} )}\le (1/\aleph) r^{-n}\|fu^2\|_{L^1(\Omega )}+r^\alpha ,\quad 0<r<r_0.
\]
This means that an estimate of the form \eqref{a15} holds with $\|f\|_{L^\infty (\Omega ^\eta )}$ substituted by 
$\|f\|_{L^\infty (\Omega _{2\eta })}$. In consequence, 
\[
\|f\|_{L^\infty (\Omega )}\le C(\|fu^2\|_{L^1(\Omega )} r^{-c}+r^\alpha ),\quad 0<r<\min (r_0 ,\delta _0) .
\]
We can then mimic the end of the proof of (1) in order to obtain the expected inequality.
\end{proof}

\section{Proof of the main results}

\begin{proof}[Proof of Theorem \ref{theorem1.1}]
Pick $\omega \Subset \Omega$. We firstly observe that, according to \eqref{1}, 
\[
\{u_V;\; V\in \mathscr{D}_0(k,v_0,\overline{V})\}\subset \mathscr{S}_w(v_0,V_0,\kappa ,M),
\]
with $\kappa =\|h\|_{L^\infty (\Gamma )}$, $V_0=v_0+\|\overline{V}\|_{L^\infty (\Omega )}$ and $M=M\left(\Omega ,p,v_0, k ,\overline{V}, h\right)$ is as in \eqref{1}. We  then apply (1) of Theorem \ref{theorem-wii} in order to obtain
\begin{equation}\label{[1]}
\|V-\tilde{V}\|_{L^\infty (\omega )}\le C\|(V-\tilde{V})u_V^2\|_{L^1(\Omega )}^{2\mu} ,
\end{equation}
with $C=C\left(\Omega ,p,v_0, k ,\overline{V}, h ,\omega \right)$ and $\mu =\mu \left(\Omega ,p,v_0, k ,\overline{V}, h,\omega\right)$.

On the other hand, we have from \cite[Theorem 2.2, page 1781]{CT} 
\begin{equation}\label{[2]}
\| I_V^{1/2}(V-\tilde{V})\|_{H^1(\Omega )}\le C\|I_{V}^{1/2}-I_{\tilde{V}}^{1/2}\|_{H^1(\Omega )}^{1/2}.
\end{equation}

Set, for simplicity's sake, $u=u_V$ (resp. $I=I_V$) and $\tilde{u}=u_{\tilde{V}}$ (resp. 
$\tilde{I}=I_{\tilde{V}}$). Then
\begin{align*}
(V-\tilde{V})u^2&=Vu^2-\tilde{V}\tilde{u}^2+\tilde{V}(u^2-\tilde{u}^2)
\\
&=I -\tilde{I}+\tilde{V}(|u|+|\tilde{u}|)(|u|-|\tilde{u}|).
\end{align*}
Hence 
\begin{equation}\label{[3]}
\|(V-\tilde{V})u^2\|_{L^1(\Omega )}\le C\left(\|I^{1/2} -\tilde{I}^{1/2}\|_{L^1(\Omega )}+\| |u|-|\tilde{u}|\|_{L^1(\Omega )}\right).
\end{equation}
But
\[
|\tilde{u}|-|u|=\frac{1}{\tilde{V}}\left(I-\tilde{I}\right)+\frac{I^{1/2}}{V\tilde{V}}\left[\left(V-\tilde{V}\right)I^{1/2}\right]
\]
implying
\begin{equation}\label{[4]}
\|(V-\tilde{V})u^2\|_{L^1(\Omega )}\le C\left(\|I^{1/2} -\tilde{I}^{1/2}\|_{L^1(\Omega )}+\|(V-\tilde{V})I^{1/2}\|_{L^1(\Omega )}\right).
\end{equation}
Now a combination of \eqref{[2]}, \eqref{[3]} and \eqref{[4]} yields
\[
\|V-\tilde{V}\|_{L^\infty (\omega )}\le C\|I^{1/2}-\tilde{I}^{1/2}\|_{H^1(\Omega )}^\mu,
\]
which is the expected inequality.
\end{proof}

\begin{proof}[Proof of Theorem \ref{theorem1.2}]
Quite similar to that of Theorem \ref{theorem1.1}. We have only to apply (2) of 
Theorem \ref{theorem-wii} instead of (1) of Theorem \ref{theorem-wii}.
\end{proof}

\begin{proof}[Proof of Theorem \ref{theorem1.2b}]
The proof relies on an improvement of a weighted stability estimate obtained in \cite[Theorem 2.2, page 1781]{CT}.

\begin{lemma}  \label{weight} Assume that the assumptions of Theorem \ref{theorem1.2b} hold. Then there exist 
$C=C(v_0, \overline V, h, K, \omega, \tilde{\Omega})>0$ and $1>\mu^\prime =\mu^\prime
(v_0, \overline V, h, K,  \omega, \tilde{\Omega})>0$ such that 
\begin{equation}\label{wei}
\| I_V^{1/2}(V-\tilde{V})\|_{L^2(\omega )}\le C\|I_{V}^{1/2}-I_{\tilde{V}}^{1/2}\|_{H^1(\tilde{\Omega})}^{\mu^\prime}.
\end{equation}
\end{lemma}

\begin{proof}
We set 
 $\theta =V^{-1/2}$ and $J=I_V^{1/2}$. We deduce from the proof 
of \cite[Theorem 2.2, page 1781]{CT} that $\theta$ satisfies
\begin{equation}\label{e25}
J\Delta\left(J \theta\right)= -\frac{J^2}{\theta} \quad  \textrm{in}\; \Omega.
\end{equation}
Referring to \cite{Ro}, we see, as $u_V$ is non identically equal to zero, that the set where $u_V$ vanishes is of zero measure. Therefore, $J=Vu_V^2$ has the same property and hence $\theta$ verifies 
\begin{equation}\label{e25b}
\Delta\left(J \theta\right)= -\frac{J}{\theta} \quad  \textrm{in}\; \Omega.
\end{equation}
Let $\tilde{\theta} =\tilde{V}^{-1/2}$ and
$\tilde{J}=I_{\tilde{V}}^{1/2}$.  Identity \eqref{e25b}, with $V$ substituted by 
$\tilde{V}$, yields
\begin{equation}\label{e27b}
\Delta \left(\tilde{J} \tilde{\theta}\right) = -
\frac{ \tilde{J}}{\tilde{\theta}} \quad  \textrm{in}\;\Omega.
\end{equation}
Taking the difference side by side of equations \eqref{e25b} and \eqref{e27b}, we obtain
\begin{align}
\Delta\left( J (\theta-
  \tilde{\theta})\right)-\frac{1}{\theta\tilde{\theta}}J(\theta-\tilde{\theta})
=  \frac{ \tilde{J}-J}{\tilde{\theta}} +\Delta\left( \tilde{\theta}(\tilde{J}-J)\right).\label{e27-25}
\end{align}
As  $J (\theta-\tilde{\theta})$ has zero Cauchy data on $\tilde{\gamma}$, we deduce from \cite[Theorem 1.7]{ARRV}  that there exists 
 $C(v_0, \overline V, h, K, \omega, \tilde{\Omega})>0$ and $1>\mu^\prime
(v_0, \overline V, h, K,  \omega, \tilde{\Omega})>0$ so that 
\begin{align}
\|J (\theta- \tilde{\theta})\|_{L^2(\omega)} \leq C\left( 
   \left\| \frac{ \tilde{J}-J}{\tilde{\theta}}\right\|_{L^2(\tilde{\Omega})} +
  \left\| \tilde{\theta}(\tilde{J}-J) \right\|_{H^1(\tilde{\Omega})}  
  \right)^{\mu^\prime}.
\end{align}
Whence the expected inequality follows.
\end{proof}

The rest of the proof is quite similar to that of Theorem \ref{theorem1.1}. We  apply again (1) of Theorem \ref{theorem-wii}  to  inequality \eqref{wei}.
\end{proof}

\begin{proof}[Proof of Theorem \ref{theorem1.3}]

We split $\gamma$ into two components $\gamma_+=\gamma \cap \Gamma_+$ and $\gamma_0=\gamma \cap \Gamma_0$.

Let $V$ and $\tilde{V}$ be as in the statement of Theorem \ref{theorem1.3}. As $p>n$, $W^{2,p}(\Omega )$ is continuously embedded in $C^1(\overline{\Omega})$. Whence $u_V^2\in W^{2,p}(\Omega )$. Inspecting the proof of \cite[Proposition 3.1]{ACT}, we get that there exists a constant $\delta =\delta \left(\Omega ,p,v_0, k ,\overline{V}, h,, \gamma_0\right)$ and a neighborhood $\mathcal{U}_0$ of $\gamma_0$ in $\omega \cup \Gamma_0$ so that $|u_V|^{-\delta}\in L^1(\mathcal{U}_0)$.
We get from the proof of \cite[Lemma 1.3]{ACT} that there exists $C_0 =C_0 \left(\Omega ,p,v_0, k ,\overline{V}, h, \gamma_0\right)$ such that
\[
\|V-\tilde{V}\|_{L^2(\mathcal{U}_0)}\le C_0\|(V-\tilde{V})u_V^2\|_{L^1(\mathcal{U}_0 )}^{\delta/(2+\delta)}.
\]
In light of \cite[Lemma B.1]{CK}, this inequality entails
\[
\|V-\tilde{V}\|_{L^\infty (\mathcal{U}_0)}\le C_0\|(V-\tilde{V})u_V^2\|_{L^1(\mathcal{U}_0 )}^{2\mu _0},
\]
for some $\mu_0= \mu_0 \left(\Omega ,p,v_0, k ,\overline{V}, h,, \gamma_0\right)$.

As in Theorem \ref{theorem1.1}, this inequality leads to the following one
\begin{equation}\label{t1.3.1}
\|V-\tilde{V}\|_{L^\infty (\mathcal{U}_0)}\le C_0\|I_V^{1/2}-I_{\widetilde{V}}^{1/2}\|_{H^1(\Omega )}^{\mu_0}.
\end{equation}

On the other hand, we easily check that 
\[
|u_V| \ge \frac{1}{2}\min_{\gamma _+}|u_V|=\frac{1}{2}\min_{\gamma _+}|h|\; (>0)
\]
in a neighborhood $\mathcal{U}_+$ of $\gamma_+$ in $\omega \cup \Gamma_+$, depending only $\Omega $, $p$, $v_0$, $k$,$\overline{V}$, $h$ and  $\gamma_+$. We apply once again \cite[Lemma B.1]{CK} in order to get
\[
\|V-\tilde{V}\|_{L^\infty (\mathcal{U}_+)}\le C_+\|(V-\tilde{V})u_V^2\|_{L^1(\mathcal{U}_+ )}^{\mu _+},
\]
for some $\mu_+= \mu_+ \left(\Omega ,p,v_0, k ,\overline{V}, h,, \gamma_+\right)$ and $C_+= C_+ \left(\Omega ,p,v_0, k ,\overline{V}, h,, \gamma_+\right)$.

From this inequality we deduce, again similarly as in the proof of Theorem \ref{theorem1.1}, 
\begin{equation}\label{t1.3.2}
\|V-\tilde{V}\|_{L^\infty (\mathcal{U}_+)}\le C_+\|I_V^{1/2}-I_{\tilde{V}}^{1/2}\|_{H^1(\Omega )}^{\mu _+}.
\end{equation}

Let $\tilde{\omega}\Subset \Omega$ so that $\omega \subset \tilde{\omega}\cup \mathcal{U}_0\cup \mathcal{U}_+$. By the interior stability estimate in Theorem \ref{theorem1.1}, there exists $\tilde{\mu}= \tilde{\mu} \left(\Omega ,p,v_0, k ,\overline{V}, h,, \omega \right)$ and $C= C\left(\Omega ,p,v_0, k ,\overline{V}, h, \omega \right)$ so that
\begin{equation}\label{t1.3.3}
\|V-\tilde{V}\|_{L^\infty (\tilde{\omega} )}\le C\|I_{V}^{1/2}-I_{\tilde{V}}^{1/2}\|_{H^1(\Omega )}^{ \tilde{\mu}}.
\end{equation}

We end up getting the expected inequality by combining \eqref{t1.3.1}, \eqref{t1.3.2} and  \eqref{t1.3.1}, with $\mu =\min (\mu_0 ,\mu _+,\tilde{\mu})$.
\end{proof}

\appendix
\section{}

In this appendix, $\Omega$ is a bounded domain of $\mathbb{R}^n$, $n\ge 2$, with Lipschitz boundary $\Gamma$. Let
\[
L=\mbox{div}(A\nabla \, \cdot )+V,
\]
where $V\in L^\infty (\Omega )$, $A=(a^{ij})$ is a symmetric matrix with  coefficients in $W^{1,\infty}(\Omega )$ and there exist $\kappa >0$ and $\Lambda >0$ so that 
\begin{equation}\label{eq1}
A(x)\xi \cdot \xi \geq \kappa |\xi |^2,\;\;  x\in \Omega , \; \xi \in \mathbb{R}^n,
\end{equation}
and 
\begin{equation}\label{eq2}
\|V\|_{L^\infty (\Omega )}+\|a^{ij}\|_{W^{1,\infty}(\Omega )}\le \Lambda ,\quad 1\leq i,j\leq n.
\end{equation}

Recall the following three-ball interpolation inequality, proved in  \cite{BCT} when $V=0$ but still holds for any bounded $V$ (see also \cite{Ch1}).

\begin{theorem}\label{theoremI}
Let $0<k<\ell<m$. There exist $C>0$ and $0<s <1$, only depending on  $\Omega$, $k$, $\ell$, $m$, $\kappa$ and $\Lambda$,  such that 
\begin{equation}\label{eq3}
\|v\|_{L^2(B(y,\ell r))}\leq C\|v\|_{L^2(B(y,kr))}^s\|v\|_{L^2(B(y,m r))}^{1-s},
\end{equation}
for all  $v\in H^1(\Omega )$ satisfying $Lv=0$ in $\Omega$, $y\in \Omega$ and $0<r< \mbox{dist}(y,\Gamma )/m$.
\end{theorem}

We know from \cite[Theorem 2.4.7, page 53]{HP} that any Lipschitz domain has the uniform interior cone condition, abbreviated to \textbf{UICP} in the sequel. In particular, there exist $R>0$ and $\theta \in \left]0,\frac{\pi}{2}\right[$ so that, to any $\tilde{x}\in \Gamma$ corresponds $\xi =\xi (\tilde{x})\in \mathbb{S}^{n-1}$ for which
\[
\mathcal{C}(\tilde{x})=\{x\in \mathbb{R}^n;\; 0<|x-\tilde{x}|<R,\; (x-\tilde{x})\cdot \xi >|x-\tilde{x}|\cos \theta \}\subset \Omega .
\]

Define the geometric distance $d_g^D$, on a bounded domain $D$ of $\mathbb{R}^n$, by
\[
d_g^D(x,y)=\inf\left \{ \ell (\psi ) ;\; \psi :[0,1]\rightarrow D \; \mbox{Lipschitz path joining}\; x \; \mbox{to}\; y\right\},
\] 
where
\[
\ell (\psi )= \int_0^1|\dot{\psi}(t)|dt
\]
is the length of $\psi$.

Note that, according to Rademacher's theorem, any Lipschitz continuous function $\psi :[0,1]\rightarrow D$ is almost everywhere differentiable with $|\dot{\psi}(t)|\le k$ a.e. $t\in [0,1]$, where $k$ is the Lipschitz constant of $\psi$.

\begin{lemma}\label{Glemma}
Let $D$ be a bounded Lipschitz domain of $\mathbb{R}^n$. Then $d_g^D\in L^\infty (D \times D )$.
\end{lemma}
A proof of this lemma can be found in \cite{CY}.

In the rest of this text
\[
\mathbf{d}_g=\|d_g^\Omega\|_{L^\infty (\Omega \times \Omega )}.
\]

\begin{proof}[Proof of Theorem \ref{theorem2.1}]
In this proof $C$ denote a generic constant that can depend only on $\Omega$, $v_0$, $V_0$, $\kappa$ and $M$.

\smallskip
\textbf{Step 1.} Let $y, y_0\in \Omega ^{3\delta}$ and $\psi :[0,1]\rightarrow \Omega$ be a Lipschitz path joining $y_0$ to $y$ so that $\ell (\psi)\le d_g^\Omega (y_0,y)+1$. Let $t_0=0$ and $t_{k+1}=\inf \{t\in [t_k,1];\; \psi (t)\not\in B(\psi (t_k),\delta )\}$, $k\geq 0$. We claim that there exists an integer $N\geq 1$ so that $\psi (1)\in B(\psi(t_N),\delta )$. If not, we would have $\psi (1)\not\in B(\psi (t_k),\delta )$ for any $k\ge 0$. As the sequence $(t_k)$ is non decreasing and bounded from above by $1$, it converges to $\hat{t}\le 1$. In particular, there exists an integer $k_0\geq 1$ so that $\psi (t_k)\in B\left(\psi (\hat{t}),\delta/2\right)$, $k\ge k_0$. But this contradicts the fact that $\left|\psi (t_{k+1})-\psi (t_k)\right| =\delta$, $k\ge 0$.

\smallskip
Let us check that $N\le N_0$, where $N_0$ only depends on $\mathbf{d}_g$ and $\delta$. Pick $1\le j\le n$ so that 
\[
\max_{1\leq i\leq n} \left|\psi _i(t_{k+1})-\psi _i(t_k)\right| =\left|\psi _j(t_{k+1})-\psi _j(t_k)\right|.
\]
Then
\[
\delta \le n\left|\psi _j (t_{k+1})-\psi _j(t_k)\right|=n\left| \int_{t_k}^{t_{k+1}}\dot{\psi}_j(t)dt\right|\le  n\int_{t_k}^{t_{k+1}}|\dot{\psi}(t)|dt  .
\]
Consequently, where $t_{N+1}=1$, 
\[
(N+1)\delta \le n\sum_{k=0}^N\int_{t_k}^{t_{k+1}}|\dot{\psi}(t)|dt=n\ell (\psi)\le n(\mathbf{d}_g+1).
\]
Therefore
\[
N\le N_0=\left[ \frac{n(\mathbf{d}_g+1)}{\delta}\right].
\]
Here $\left[ n(\mathbf{d}_g+1)/\delta\right]$ is the integer part of $n(\mathbf{d}_g+1)/\delta$.

Let $y_k=\psi (t_k)$, $0\le k\le N$.  If $|z-y_{k+1}|<\delta$ then 
\[ |z-y_k|\le |z-y_{k+1}|+|y_{k+1}-y_k|<2\delta. \] 
In other words, $B(y_{k+1},\delta )\subset B(y_k,2\delta)$.

\smallskip
We get from Theorem \ref{theoremI}
\begin{equation}\label{est1}
\|u\|_{L^2(B(y_j,2\delta ))}\leq C\|u\|_{L^2(B(y_j,3\delta ))}^{1-s}\|u\|_{L^2(B(y_j,\delta ))}^s,\quad 0\leq j\leq N.
\end{equation}

\smallskip
Set $I_j=\|u\|_{L^2(B(y_j,\delta ))}$, $0\le j\le N$ and $I_{N+1}=\|u\|_{L^2(B(y,\delta ))}$. Since $B(y_{j+1},\delta )\subset B(y_j,2\delta )$, $1\le j\le N-1$, estimate \eqref{est1} implies
\begin{equation}\label{est2}
I_{j+1}\le C M_0^{1-s}I_j^s,\quad 0\le j\le N,
\end{equation}
where we set $M_0=\|u\|_{L^2(\Omega )}$.

\smallskip
Let $C_1=C^{1+s+\ldots +s^{N+1}}$ and $\beta =s^{N+1}$. Then, by a simple induction argument, estimate \eqref{est2} yields
\begin{equation}\label{est3}
I_{N+1}\leq C_1M_0^{1-\beta}I_0^\beta .
\end{equation}

Without loss of generality, we assume in the sequel that $C\ge 1$ in \eqref{est2}.
Using that $N\leq N_0$, we have 
\begin{align*}
&\beta \ge \beta _0=s^{N_0+1}\ge se^{-\frac{\kappa}{\delta}}=\psi (\delta ),\;\;  \mbox{where}\; \kappa =n(\mathbf{d}_g+1)|\ln s|,
\\
&C_1\le C^{\frac{1}{1-s}},
\\
&\left(\frac{I_0}{M_0}\right)^\beta\le \left(\frac{I_0}{M_0}\right)^{\beta_0}.
\end{align*}
These estimates in \eqref{est3} gives
\[
\frac{I_{N+1}}{M_0}\leq C\left(\frac{I_0}{M_0}\right)^{\psi (\delta )}.
\]
In other words,
\[
\frac{\|u\|_{L^2(B(y,\delta ))}}{\|u\|_{L^2(\Omega )}}\le C\left(\frac{\|u\|_{L^2(B(y_0,\delta ))}}{\|u\|_{L^2(\Omega )}}\right)^{\psi (\delta )}.
\]
Applying Young's inequality, we get from this inequality
\begin{equation}\label{est4}
\|u\|_{L^2(B(y,\delta ))}\le C\left( \epsilon^{\frac{1}{1-\psi (\delta )}}\| u\|_{L^2(\Omega )}+\epsilon^{-\frac{1}{\psi (\delta )}}\|u\|_{L^2(B(y_0,\delta ))}\right),
\end{equation}
$\epsilon >0$,  $y,y_0\in \Omega^{3\delta}$.

\textbf{Step 2.} Fix $\tilde{x}\in \Gamma$ so that $|u(\tilde{x})|=\|u\|_{L^\infty (\Gamma )}$. Let $\xi =\xi (\tilde{x})$ be as in the definition of the \textbf{UICP}. Let $x_0=\tilde{x}+\delta \xi$, $\delta \le R/2$, $d_0=|x_0-\widetilde{x}|=\delta $ and $\rho_0=d_0\sin \theta/3$. Note that $B(x_0,3\rho_0)\subset \mathcal{C}(\tilde{x})$.

\smallskip
By induction in $k$, we construct a sequence of balls $(B(x_k, 3\rho _k))$, contained in $\mathcal{C}(\tilde{x})$, as follows
\begin{eqnarray*}
\left\{
\begin{array}{ll}
x_{k+1}=x_k-\alpha _k \xi ,
\\
\rho_{k+1}=\mu \rho_k ,
\\
d_{k+1}=\mu d_k,
\end{array}
\right.
\end{eqnarray*}
where
\[
d_k=|x_k-\tilde{x}|,\;\; \rho _k=\vartheta d_k,\;\; \alpha _k=(1-\mu)d_k ,
\]
with
\[
\vartheta=\frac{\sin \theta}{3},\quad \mu =\frac{3-2\sin\theta}{3-\sin\theta} .
\]
Note that this construction guarantees that, for each $k$, $B(x_k,3\rho _k)\subset \mathcal{C}(\tilde{x})$ and
\begin{equation}\label{3.101}
B(x_{k+1},\rho _{k+1})\subset B(x_k,2\rho _k).
\end{equation}

We get, by applying Theorem \ref{theoremI}, that there exist $C>0$ and $0<s<1$, only depending on $\Omega$, $v_0$ and $V_0$, so that
\begin{align}
 \| u\|_{L^2(B(x_k,2\rho _k))}&\le C\|u\|_{L^2(B(x_k,3\rho _k))}^{1-s}\|u\|_{L^2(B(x_k,\rho _k))}^s\label{3.102}
\\
& \le CM^{1-s}\|u\|_{L^2(B(x_k,\rho _k))}^s.\nonumber
\end{align}
In light of \eqref{3.101}, \eqref{3.102} gives
\begin{equation}\label{3.103}
\|u\|_{L^2(B(x_{k+1},\rho _{k+1}))}\le CM^{1-s}\|u\|_{L^2(B(x_k,\rho _k))}^s.
\end{equation}

Let $J_k=\|u\|_{L^2(B(x_k,\rho _k))}$, $k\ge 0$. Then \eqref{3.103} is rewritten as follows 
\[
J_{k+1}\le CM^{1-s}J_k^s.
\]
An induction in $k$ yields
\[
J_k\le C^{1+s +\ldots +s^{k-1}}M^{(1-s)(1+s +\ldots +s^{k-1})}J_0^{s^k}.
\]

That is
\begin{equation}\label{3.104}
J_k\le \left[C^{\frac{1}{1-s}}M\right]^{1-s^k}J_0^{s^k}.
\end{equation}
Applying Young's inequality we obtain, for any $\epsilon>0$,

\begin{align}
J_k&\le (1-s^k)\epsilon^{\frac{1}{1-s^k}}C^{\frac{1}{1-s}}M+s^k\epsilon^{-\frac{1}{s^k}}J_0 \label{3.105}
\\
& \le \epsilon^{\frac{1}{1-s^k}}C^{\frac{1}{1-s}}M+\epsilon^{-\frac{1}{s^k}}J_0\nonumber
\\
& \le C\epsilon^{\frac{1}{1-s^k}}M+\epsilon^{-\frac{1}{s^k}}J_0. \nonumber 
\end{align}

Now, since $u\in C^{0,1/2}(\overline{\Omega})$,
\[
|u(\tilde{x})|\le [u]_{1/2} |\tilde{x}-x|^{1/2} +|u(x)|,\quad x\in B(x_k,\rho _k).
\]

Hence
\[
|\mathbb{S}^{n-1}|\rho_k^n|u(\tilde{x})|^2\le 2[u]_{1/2}^2\int_{B(x_k,\rho _k)} |\tilde{x}-x|dx+ 2\int_{B(x_k,\rho _k)} |u(x)|^2dx.
\]
Or equivalently 
\[
|u(\tilde{x})|^2\le 2|\mathbb{S}^{n-1}|^{-1}\rho_k^{-n}\left([u]_{1/2}^2\int_{B(x_k,\rho _k)} |\tilde{x}-x|dx+ \int_{B(x_k,\rho _k)} |u(x)|^2dx\right).
\]
A simple computation shows that $d_k=\mu ^kd_0$. Then 
\[
|\tilde{x}-x|\leq |\tilde{x}-x_k|+|x_k-x|\le d_k+\rho_k=(1+\vartheta)d_k=(1+\vartheta )\mu ^kd_0.
\]
Therefore,
\[
|u(\tilde{x})|^2\le 2\left(M^2(1+\vartheta )^{1/2} d_0^{1/2}\mu ^{k}+ |\mathbb{S}^{n-1}|^{-1}(\vartheta d_0)^{-n}\mu ^{-nk}\|u\|_{L^2(B(x_k,\rho _k))}^2\right)
\]
implying, when $d_0 (=\delta ) \le 1$,
\begin{equation}\label{3.106}
|u(\tilde{x})|\le C\left(M\mu ^{k/2}+\mu ^{-nk/2}\delta ^{-n/2}J_k\right).
\end{equation}

Inequalities  \eqref{3.105} and \eqref{3.106} gives
\begin{equation}\label{3.107}
|u(\tilde{x})|\le C\left(\mu ^{k/2}M+\mu ^{-nk/}\epsilon^{1/(1-s^k)}\delta ^{-n/2}M+\mu ^{-nk/2}\epsilon^{-1/s^k}\delta^{-n/2}J_0\right).
\end{equation}
We get, by choosing $\epsilon=\mu^{(1-s^k)(n+1)k/2}$ in \eqref{3.107},
\[
|u(\tilde{x})|\le C\left(\mu ^{k/2}M+ \mu ^{k/2}\delta ^{-n/2}M+ \mu ^{-(n+1)k/(2s^k)+k/2}\delta^{-n/2}J_0\right).
\]
Hence
\begin{equation}\label{3.108}
|u(\tilde{x})|\le C\delta ^{-n/2}\left(\mu ^{k/2}M+ \mu ^{-(n+1)k/(2s^k)}J_0\right),
\end{equation}
by using  $\delta \le \mbox{diam}(\Omega )$.

Let $t>0$ and $k$ be the integer so that $k\le t<k+1$. It follows from \eqref{3.108}
\begin{equation}\label{3.109}
|u(\tilde{x})|\leq C\delta ^{-n/2}\left(\mu ^{t/2}M+\mu ^{-t(n+1)/(2s^t)}J_0\right).
\end{equation}
Let $p =(n+1)/2+|\ln s|$. Then \eqref{3.109} yields
\begin{equation}\label{3.110}
|u(\tilde{x})|\leq C\delta ^{-n/2}\left(\mu ^{t/2}M+\mu ^{-e^{p t}}J_0\right).
\end{equation}
Putting  $e^{p t}=1/\epsilon$, $0<\epsilon <1$, we get from \eqref{3.110} 
\begin{equation}\label{3.111}
|u(\tilde{x})|\leq C\delta ^{-n/2}\left(\epsilon ^{\beta}M+e^{|\ln \mu |/\epsilon} J_0\right),
\end{equation}
where $\beta = |\ln \mu |/(2p)$.

\textbf{Step 3.} A combination of \eqref{est4} and \eqref{3.111} entails, with $0<\epsilon <1$ and $\epsilon_1>0$,
\[
|u(\tilde{x})|\le C\delta ^{-n/2}\left(\epsilon ^{\beta}M+e^{|\ln \mu |/\epsilon} \left( \epsilon_1^{1/(1-\psi (\delta ))}M+\epsilon_1^{-1/\psi (\delta )}\|u\|_{L^2(B(y_0,\delta ))}\right)\right).
\]
Hence, where $\ell =n/2$ and $\rho =|\ln \mu |$,
\[
\|u\|_{L^\infty (\Gamma )}\le C\delta^{-\ell}\left(\epsilon ^{\beta}M+e^{\rho /\epsilon} \left( \epsilon_1^{1/(1-\psi (\delta ))}M+\epsilon_1^{-1/\psi (\delta )}\|u\|_{L^2(B(y_0,\delta ))^n}\right)\right).
\]
In this inequality we take $\epsilon_1=\epsilon^{\beta (1-\psi (\delta ))}e^{-\rho(1-\psi (\delta ))/\epsilon}$. Using that 
\[
\epsilon _1^{-1/\psi(\delta )}\le e^{(\rho +\beta )(1-\psi(\delta ))/(\epsilon \psi(\delta ))},
\]
we obtain in a straightforward manner
\[
\|u\|_{L^\infty (\Gamma )}\le C\delta^{-\ell}\left(\epsilon ^{\beta}M+e^{(\rho +\beta )(1-\psi(\delta ))/(\epsilon \psi(\delta ))}\|u\|_{L^2(B(y_0,\delta ))^n}\right).
\]
If $\phi(\delta )=(\rho +\beta)(1-\psi(\delta ))/\psi(\delta )$ then we can rewrite the previous estimate as follows
\[
\|u\|_{L^\infty (\Gamma )}\le C\delta^{-\ell}\left(\epsilon ^{\beta}M+e^{\phi (\delta )/\epsilon}\|u\|_{L^2(B(y_0,\delta ))^n}\right),
\]
or equivalently 
\begin{equation}\label{est5}
\|u\|_{L^\infty (\Gamma )}\le C\delta^{-\ell}\left(t^{-\beta}M+e^{t\phi (\delta )}\|u\|_{L^2(B(y_0,\delta ))}\right),\;\; t>1.
\end{equation}
If $M/\|u\|_{L^2(B(y_0,\delta ))}> e^{\phi (\delta )}$, we find $t>1$ so that $M/\|u\|_{L^2(B(y_0,\delta ))}=t^\beta e^{t\phi (\delta )}$. The estimate \eqref{est5} with that $t$ yields
\[
\|u\|_{L^\infty (\Gamma )}\le C\delta^{-\ell}M\left(\frac{1}{(\beta +\phi(\delta )}\ln\left(\frac{M}{\|u\|_{L^2(B(y_0,\delta ))}} \right)\right)^{-\beta} .
\]
In light of the inequality $\|u\|_{L^\infty (\Gamma )}\ge \eta$, this estimate implies
\[
\eta\le C\delta^{-\ell}M\left(\frac{1}{\beta +\phi(\delta )}\ln\left(\frac{M}{\|u\|_{L^2(B(y_0,\delta ))}} \right)\right)^{-\beta} .
\]
This inequality is equivalent to the following one
\begin{equation}\label{est6}
Me^{-C(\beta +\phi(\delta ))\delta^{-\ell /\beta}\left(M/\eta\right)^{1/\beta}}\le \|u\|_{L^2(B(y_0,\delta ))}.
\end{equation}

Otherwise,
\begin{equation}\label{est7}
Me^{-\phi (\delta )}\le \|u\|_{L^2(B(y_0,\delta ))}.
\end{equation}

We derive from \eqref{est6} and \eqref{est7} that, there exist $C>0$ and $\delta ^\ast$ so that 
\[
e^{-e^{c/\delta}}\le \|u\|_{L^2(B(y_0,\delta ))^n},\; 0< \delta \le \delta ^\ast.
\]
Obviously, a similar estimate holds for $\delta \ge \delta ^\ast$.
\end{proof}

\end{document}